\documentclass[11pt]{amsart}
\usepackage[utf8]{inputenc}
\usepackage{amsfonts}
\usepackage{mathtools}
\usepackage{amsmath}
\usepackage{xcolor}
\usepackage{amsthm}
\usepackage{pdflscape}
\usepackage{pgfplots}
\usepackage{mathrsfs}
\usepackage{hyperref}
\usepackage[pagewise]{lineno}

\newtheorem{thm}{Theorem}[section]
\newtheorem{cor}[thm]{Corollary}

\newtheorem{lem}[thm]{Lemma}
\newtheorem{conj}[thm]{Conjecture}

\newtheorem{fact}[thm]{Fact}

\theoremstyle{definition}
\newtheorem{defn}[thm]{Definition}

\newtheorem{exmp}[thm]{Example}

\newtheorem{notn}[thm]{Notation}

\newcommand{\gen}[1]{\left\langle#1\right\rangle}
\theoremstyle{remark}
\newtheorem{rem}[thm]{Remark}

\theoremstyle{plain}
\newtheorem{thmx}{Theorem}

\begin{document}
\title[Algebraic independence and the classical Lotka-Volterra system]{Algebraic independence of the solutions of the classical Lotka-Volterra system}
\author{Yutong Duan}
\address{Yutong Duan\\ Department of Mathematics, Statistics, and Computer Science\\
	University of Illinois Chicago\\
	322 Science and Engineering Offices (M/C 249)\\
	851 S. Morgan Street\\
	Chicago, IL 60607-7045}
\email{yduan24@uic.edu}
\author{Joel Nagloo}
\address{Joel Nagloo\\ Department of Mathematics, Statistics, and Computer Science\\
	University of Illinois Chicago\\
	322 Science and Engineering Offices (M/C 249)\\
	851 S. Morgan Street\\
	Chicago, IL 60607-7045}
\email{jnagloo@uic.edu}
\thanks{\today. J. Nagloo and Y. Duan are partially supported by NSF grants DMS-2203508 and DMS-2348885. Some of the work in this paper was completed while J. Nagloo was supported by an AMS Centennial Fellowship.}

\maketitle

\begin{abstract}
\sloppy Let $(x_1,y_1),\ldots,(x_n,y_n)$ be distinct non-constant and non-degenerate solutions of the classical Lotka-Volterra system
\begin{equation}\notag
\begin{split}
    x'= axy + bx \\
    y'= cxy + dy\\
\end{split}
\end{equation}where $a,b,c,d\in\mathbb{C}\setminus\{0\}$. We show that if $d$ and $b$ are linearly independent over $\mathbb{Q}$, then the solutions are algebraically independent over $\mathbb{C}$, that is $tr.deg_{\mathbb{C}}\mathbb{C}(x_1,y_1,\ldots,x_n,y_n)=2n$. As a main part of the proof, we show that the set defined by the system in universal differential fields, with $d$ and $b$ linearly independent over $\mathbb{Q}$, is strongly minimal and geometrically trivial. Our techniques also allows us to obtain partial results for some of the more general $2d$-Lotka-Volterra system.
\end{abstract}

\section{Introduction}
In this paper, we initiate the study of the two dimensions ($2d$) Lotka-Volterra system

\begin{equation}\label{gLV}
\begin{dcases}
    x'= x(a_1x + a_2y + a_3) \\
    y'= y(b_1x+b_2y+b_3),\\
\end{dcases}
\end{equation}
where $a_1,a_2,a_3,b_1,b_2,b_3\in\mathbb{C}$, in the context of model theory. This quadratic system was first studied in the early 1900's independently by Lotka \cite{Lotka} and Volterra \cite{Volterra} in the special case of what we have called the classical Lotka-Volterra system. Initially appearing in ecology, where it was showed to model two species (predator and prey) in competition, it has since appeared in several other areas such as chemistry and physics (cf. \cite{AhmadStamova}, \cite{LeconteMassonQi}, \cite{Hering}), although we do not focus on those applications. The $2d$-Lotka-Volterra system has also been extensively studied in the framework of integrability. We direct the reader to the work of Cair\'o, Giacomini and Llibre \cite{CairoGiacominiLlibre} where a complete classification of all parameters such that the $2d$-Lotka–Volterra system has a Liouvillian first integral is given, and where the history of this problem - including an extensive literature review - can be found.

In the context of model theory, using techniques from geometric stability (which they further developed),  Eagles and Jimenez \cite{EaglesJimenez}, have recently showed that the non-constant and non-degenerate solutions of the classical Lotka-Volterra system
\begin{equation}
\begin{dcases}
    x'= axy + bx \\
    y'= cxy + dy\\
\end{dcases}
\end{equation}
are almost never Liouvillian. By non-degenerate solutions of the $2d$-Lotka–Volterra system (\ref{gLV}), we mean solutions $(x,y)$ of the system so that neither $x=0$ nor $y=0$. Indeed, it is not hard to see that the systems 
\begin{equation}\notag
\begin{dcases}
    x= 0 \\
    y'= y(b_2y+b_3),\\
\end{dcases}
\end{equation}
and 
\begin{equation}\notag
\begin{dcases}
    x'= x(a_1x + a_3) \\
    y= 0,\\
\end{dcases}
\end{equation}
are subsystems of the $2d$-Lotka–Volterra system. 

In essence, our first result gives that (in the strongest possible algebraic terms) for the classical Lotka-Volterra system, when $\frac{d}{b}\not\in\mathbb{Q}$, these are the only subsystems. In other words, we show that the system, including the non-degeneracy conditions, is strongly minimal. Recall, that a planar system of algebraic differential equations
\begin{equation}\notag
\begin{dcases}
    x'= f(x,y) \\
    y'= g(x,y)\\
    x\neq0 \\
    y\neq0,\\
\end{dcases}
\end{equation}
where $f,g\in\mathbb{C}[x,y]$, is {\em strongly minimal} if for any differential field $K$ extending $\mathbb{C}$ and any solution $(x,y)$ in some differential field extension of $K$, \begin{equation}\label{StrongMinimality}tr.deg_KK(x,y)=0\text{ or }2.\end{equation}
This notion, which is fundamental in model theory, is closely related to irreducibility as defined by Umemura \cite{Umemura} in his study of the Painlevé equations. Roughly speaking, given the class of {\em classical functions}, an autonomous planar system of differential equations is said to be irreducible if none of its solutions lie in that class and none satisfy any first-order differential equation; hence all of its solutions define {\em new functions}. In other words, the irreducibility of a system of differential equations means that the latter cannot be studied by reducing to simpler forms of solutions. The idea of formalizing this notion originated in the work of Painlev\'e \cite{Painleve} in the early 1900's and was described in modern mathematical language in the above-mentioned work of Umemura and later in the work of the second author with Pillay \cite{NaglooPillay}: an automous planar system of differential equations is irreducible if and only if it is strongly minimal. We also mention related work of Casale \cite{Casale} where the question of irreducibility is defined and studied using Malgrange's approach to differential Galois theory.

Strong minimality has also been instrumentally used to obtain several functional transcendence results for solutions of algebraic differential equations (cf. \cite{Nagloo}), including recent work on the Ax-Lindemann-Weierstrass and Ax-Schanuel conjectures for automorphic functions (cf. \cite{BlazquezCasaleFreitagNagloo} and \cite{CasaleFreitagNagloo}). In this paper, we successfully carry out this strategy - of using strong minimality to study the existence of algebraic relations between solutions - for the classical Lotka-Volterra system. Our first results are as follows.
\begin{thmx}\label{Theorem1}
Let $a,b,c,d\in\mathbb{C}^*$ and assume that $d$ and $b$ are linearly independent over $\mathbb{Q}$. Then the classical Lotka-Volterra system
\begin{equation}\notag
\begin{dcases}
    x'= axy + bx \\
    y'= cxy + dy\\
     x\neq0 \\
    y\neq0\\
\end{dcases}
\end{equation}
is strongly minimal.
\end{thmx}
We obtain a similar result for special cases of the $2d$-Lotka–Volterra system.

\begin{thmx}\label{Theorem2}
\sloppy Let $a,b,c,d\in\mathbb{C}^*$. Then the 2d-Lotka-Volterra system
\begin{equation}\notag
\begin{dcases}
    x'= x(ay + b) \\
    y'= y(cx + dy)\\
     x\neq0 \\
    y\neq0\\
\end{dcases}
\end{equation}
is strongly minimal.
\end{thmx}
Since the system is symmetric in the variables $x$ and $y$, one also immediately see that Theorem \ref{Theorem2} holds for the $2d$-Lotka-Volterra system (\ref{gLV}) with $a_3=b_2=0$. We will not focus on such distinctions. Our proof combines in quite a novel way the new techniques developed in the work of the second author, joint with Freitag, Jaoui and Marker \cite{FreitagJaouiMarkerNagloo}, using Puiseux series to prove strong minimality of some second order algebraic differential equations, with that of Freitag, Jaoui and Moosa \cite{FreitagJaouiMoosa} on the degree of non-minimality. In the latter work, it is shown that in order to prove that an autonomous algebraic differential equations (or systems) is strongly minimal, one only has to check that condition (\ref{StrongMinimality}) holds for a differential field, $K$, generated by one solution of the system.

As noted above,  establishing strong minimality has strong consequences on the questions of existence of algebraic relations between the solutions of the system. Indeed, one immediately obtains, using \cite[Proposition 5.8]{CasaleFreitagNagloo}, that the systems given in Theorem \ref{Theorem1} and Theorem \ref{Theorem2} are geometrically trivial (see Definition \ref{GeometricTriviality}): if any distinct pair of solutions is algebraically independent over $\mathbb{C}$ (indeed over any differential field extension), then so are any finite number of distinct solutions. Combining this with an old striking result of Brestovski \cite[Theorem 2]{Brestovski} for certain order two differential equations, we obtain the main result stated in the abstract.

\begin{thmx}\label{Theorem3}
 \sloppy Let $a,b,c,d\in\mathbb{C}^*$ and assume that $d$ and $b$ are linearly independent over $\mathbb{Q}$.  Assume that $(x_1,y_1),\ldots,(x_n,y_n)$ are distinct non-constant solutions, in any differential field extension\footnote{For example one can take the field of meromorphic functions on some open disc $D\subset\mathbb{C}$.} of $\mathbb{C}$, of the classical Lotka-Volterra system
    \begin{equation}\notag
\begin{dcases}
    x'= axy + bx \\
    y'= cxy + dy\\
     x\neq0 \\
    y\neq0.\\
\end{dcases}
\end{equation}
Then $tr.deg_{\mathbb{C}}\mathbb{C}(x_1,y_1,\ldots,x_n,y_n)=2n$, that is the solutions are algebraically independent over $\mathbb{C}$. 
\end{thmx}
The paper is organized as follows. In Section \ref{Section2}, we describe some of the basics of the model theory of differentially closed fields. In particular we give all the tools, from that set up, that will be needed for our proofs. We also prove that the $\omega$-categoricity conjecture holds for all order $2$ strongly minimal differential equations of Brestovski form. In Section \ref{Section3}, we prove Theorem \ref{Theorem3} first, assuming Theorem \ref{Theorem1}, and then proceed to prove the latter. We end with the proof of Theorem \ref{Theorem2}.

\subsection*{Acknowledgement}
As we were finishing up writing this paper, we learned that R\'emi Jaoui also recently realized that some completely different (but related) techniques can be used to prove Theorem \ref{Theorem3} and to show that the case $b=d$ is the only obstruction to $LV_{a,b,c,d}$ being strongly minimal. It is not yet clear to any of us whether these ideas could be extended to the 2d-Lotka-Volterra system. The first author, Eagles and Jimenez have recently explored how the ideas in the current paper (especially Brestovski's approach) and those of Jaoui could be used to extend the above results, including studying the algebraic relations between the solutions of {\em multiple} Lotka-Volterra systems. Their results can now be found in \cite{DuanEaglesJimenez}. We thank the anonymous referee for important comments and suggestions. We also thank James Freitag for useful discussions around the $\omega$-categoricity of equations of Brestovski form.

\section{Preliminaries}\label{Section2}
In this paper, it will be convenient and important to work in $(\mathcal{U},D)$, a ``universal differential field'' in the sense of Kolchin. So we begin by recalling some of the main properties of such a field $\mathcal{U}$. We restrict our attention to those properties and notations that will be needed in the current paper. We direct the reader to \cite{Marker} for a more comprehensive view, including discussions concerning the existence of $(\mathcal{U},D)$. 

First, we require that $(\mathcal{U},D)$  is a differential field of characteristic zero. So $\mathcal{U}$ is a field of characteristic $0$ and $D:\mathcal{U}\rightarrow \mathcal{U}$ is a derivation, i.e., an additive homomorphism that satisfies the Leibniz rule $D(ab)=bD(a)+aD(b)$. By a differential subfield $K$ of $\mathcal{U}$ we mean a differential field $(K,\partial)$ such that $K$ is a subfield of $\mathcal{U}$ and such that the restriction $D_{|K}=\partial$. We will simply write $D$ for $\partial$, i.e., we will keep the same notation for both derivations. We also say that $\mathcal{U}$ is a differential field extension of $K$.
\begin{notn}
\sloppy For a differential subfield $K\subseteq\mathcal{U}$, we denote by $K\{X\}=K[X,DX,D^2X,\ldots]$ the ring of differential polynomials in variable $X$ and by $K\gen{X}$ the field of differential rational function, i.e., the quotient field of $K\{X\}$. We will write $K^{alg}$ for the algebraic closure of $K$ in the usual algebraic sense.
\end{notn}

We also require that $\mathcal{U}$ is {\em differentially closed}, namely, if a finite system of differential polynomial equations in $\mathcal{U}\{X\}$ has a solution in a differential field extension of $\mathcal{U}$, then it already has one in $\mathcal{U}$. Finally we assume that  $\mathcal{U}$ is saturated of cardinality the continuum and so we can (and will) identify its field of constants $C_\mathcal{U}=\{u\in\mathcal{U}:Du=0\}$ with the field of complex numbers $\mathbb{C}$. Furthermore, it also implies that any finitely generated (over $\mathbb{Q}$) differential field can be embedded in $\mathcal{U}$. In this sense, $\mathcal{U}$ will be a universal domain in which all our solutions, parameters/coefficients and subfields live. So from now on, we fix such a universal differential field  $(\mathcal{U},D)$. 

Given $K$ a differential field (so a differential subfield of $\mathcal{U}$), by a {\em definable set\footnote{Recall that $DCF_0$ has quantifier elimination.} over $K$}, we mean a finite Boolean combination of solution sets, in $\mathcal{U}$, of differential polynomial equations in $K\{X\}$. When we only say ``definable set'' (without mentioning a specific differential field), we will always mean definable over $\mathcal{U}$ and so over some unspecified differential subfield of $\mathcal{U}$. If $x\in{\mathcal U}$, we denote by $K\gen{x}$ the differential field generated by $x$, i.e., the field $K(x,x',x'',\ldots)$, where $x'$ is short for $D(x)$.

\subsection{Autonomous planar systems} Let $\mathcal{C}$ be a finitely generated (over $\mathbb{Q}$) algebraically closed subfield of $\mathbb{C}$. Recall that by a planar system over $\mathcal{C}$ we mean a system of differential equations of the form
\begin{equation}\notag
S=\begin{dcases}
    X'= f(X,Y) \\
    Y'= g(X,Y)\\
\end{dcases}
\end{equation}
where $f,g\in\mathcal{C}[X,Y]$ are polynomials in two variables. In the rest of the paper, we will more carefully distinguish the notation for a solution $(x,y)$ of a system, written in lowercase $x$ and $y$,  and the notation for the variables $X$ and $Y$ used to define the systems. 

Given such a system $S$, for any differential field extension $K$ of $\mathcal{C}$ we obtain a well-defined derivation $D_S$ on $K[X,Y]$ (and even on $K(X,Y)$)
\[D_S:= D+f(X,Y)\frac{\partial}{\partial X}+g(X,Y)\frac{\partial}{\partial Y}.\]
So if $P\in K[X,Y]$, then $D_SP=P^D+f(X,Y)\frac{\partial P}{\partial X}+g(X,Y)\frac{\partial P}{\partial Y}$, where $P^D$ is the polynomial obtained by differentiating the coefficients of $P$.
\begin{defn}
A polynomial $P\in K[X,Y]$ is said to be $D_S$-invariant (or just invariant when the context is clear) if there is a polynomial $Q\in K[X,Y]$ such that $$D_SP=QP.$$
\end{defn}
If $F$ is a $D_S$-invariant polynomial and $I=(F)$, the ideal generated by it in $K[X,Y]$, satisfies $0\neq I\subset K[X,Y]$, then we call the zero locus $V(I)$ of $I$ a {\em $D_S$-invariant curve over $K$}. These can be seen as the subsystems of the original planar system over $K$.
Following the discussion in the introduction, unsurprisingly, we have the following example.
\begin{exmp}
An easy computation shows that $P:=XY$ is an invariant polynomial for the classical Lotka-Volterra system
$$X(aY+b)\frac{\partial P}{\partial X}+Y(cX+d)\frac{\partial P}{\partial Y}=(aY+cX+b+d)P.$$
It also is one for the 2d-Lotka-Volterra system.
\end{exmp}
We have the following key notion.
\begin{defn}\label{PlanarStronglyMinimal}
Let $\mathcal{S}:=\{(x,y)\in\mathbb{A}^2\;:\;x'=f(x,y)\text{ and } y'=g(x,y) \}$ be the solution of some planar system over $\mathcal{C}$. We say that $\mathcal{S}$ is {\em strongly minimal} if for any differential field extension $K$ of $\mathcal{C}$, one of the following equivalent conditions holds
\begin{enumerate}
\item For any solution $(x,y)\in\mathcal{S}$
\[tr.deg_KK(x,y)=0\text{ or }2,\]
\item There are no $D_S$-invariant curve defined over $K$.
\end{enumerate}
\end{defn}
We direct the reader to \cite[Section 2.2]{NaglooPillay} where a proof of the above equivalence is given. Moreover, it is not hard to see that if $tr.deg_KK(x,y)=1$, then there is a non-zero polynomial $F\in K[X,Y]$ such that $F(x,y)=0$. This polynomial is $D_S$-invariant.
\begin{rem} {\color{white}Test}
\begin{enumerate}
\item We will usually say that the system itself is strongly minimal. This will of course mean that its set of solutions is strongly minimal.
\item The second equivalent condition is usually referred to as Umemura's Condition (J).
\item If $tr.deg_KK(x,y)=2$, we will say that $(x,y)$ is a $K$-generic solution of the system $S$.
\item Given a planar system $S$ over $\mathcal{C}$, suppose that one already has finitely many $D_S$-invariant polynomials $F_1,\ldots ,F_n$ over $\mathcal{C}$. Then one can ask whether these give rise to the only $D_S$-invariant curves (over any $K$). This is equivalent to asking whether the system, together with the conditions $F_i(X,Y)\neq 0$, is strongly minimal.
\end{enumerate}
\end{rem}
As mentioned in the Introduction, Umemura showed that strong minimality (i.e, Condition (J)) is a key concept that can be used to determine whether a system has classical solutions (cf. \cite{Umemura} and \cite{UmemuraWatanabe}). However, it is when placed in the context of model theory that the role that strong minimality can play on the problem of classifying the algebraic relations among solutions of the system is revealed. The following is the model theoretic definition of strong minimality.
\begin{defn}\label{ClassicalStrongMinimal}
A definable set $\mathcal{X}\subset\mathcal{U}$ is said to be strongly minimal if it is infinite and for every definable subset $\mathcal{Y}$ of $\mathcal{X}$, either $\mathcal{Y}$ or $\mathcal{X} \setminus \mathcal{Y}$ is finite.
\end{defn}
Given a $\mathcal{C}$-generic solution $(x,y)$ of some planar system $S$ over $\mathcal{C}$, the field $\mathcal{C}(x,y)$ is obviously a differential field and by definition has transcendence degree $2$ over $\mathcal{C}$. In particular $x'\in \mathcal{C}(x,y) $ and hence it follows that $\mathcal{C}\gen{x}=\mathcal{C}(x,y)$. So $x''\in \mathcal{C}\gen{x}$, meaning that there is a rational function $h\in \mathcal{C}(X,Z)$ such that $x''=h(x,x')$. Of course, one can do the same for $y$. It is not hard to see that the set of solutions of the equation $X''=h(X,X')$ is strongly minimal according to Definition \ref{ClassicalStrongMinimal} if and only if the planar system $S$ is strongly minimal according to Definition \ref{PlanarStronglyMinimal}. Indeed, the sets defined by the two will be in definable bijection, so that one can study the differential algebraic properties of one of the definable set by studying those of the other. It turns out that strongly minimal autonomous planar systems (and more generally $\mathbb{C}$-definable strongly minimal sets) can only have at most a ``binary structure''. Let us make this more precise.
\begin{defn}\label{GeometricTriviality}
Let $\mathcal{X}$ be a strongly minimal set definable over some differential field $K$. We say that $\mathcal{X}$ is {\em geometrically trivial} if for any differential field extension $F$ of $K$ and any finite collection of distinct $x_1,\ldots,x_n\in \mathcal{X}\setminus F^{alg}$, if $x_1,\ldots,x_n$ together with their derivatives\footnote{The strong minimality assumption guarantees that the differential fields $F\gen{x_i}$ has finite transcendence degree over $F$ and hence one has to only consider a finite number of derivatives.}  are algebraically dependent over $F$, then for some $i<j$ we have that $x_i,x_j$ together with their derivatives are algebraically dependent over $F$.
\end{defn}
We have the following fundamental result, a proof of which can be found in \cite[Proposition 5.8]{CasaleFreitagNagloo}.
\begin{fact}\label{SMimpliesGT}
Let $\mathcal{X}$ be a strongly minimal set and assume that $\mathcal{X}$ is definable over $\mathbb{C}$. Assume that $\text{ord}(\mathcal{X})>1$, where $$\text{ord}(\mathcal{X})=sup\{tr.deg_KK\gen{x}:x\in\mathcal{X}\text{ and $\mathcal{X}$ is definable over $K$}\}.$$Then $\mathcal{X}$ is geometrically trivial.
\end{fact}
\sloppy Hence, if we show that a planar system $S$ defined over $\mathcal{C}$ is strongly minimal, then it is geometrically trivial, that is if $F$ is a differential field extension of $\mathcal{C}$ and $(x_1,y_1),\ldots,(x_n,y_n)$ are distinct non-constant solutions such that $tr.deg_FF(x_1,y_1,\ldots,x_n,y_n)<2n$, then for some $i<j$, we have that $tr.deg_FF(x_i,y_i,x_j,y_j)=2$.
\begin{rem}\label{WeaklyOrthogonal}
We get a stronger conclusion: if $tr.deg_FF(x_1,y_1,\ldots,x_n,y_n)<2n$, then for some $i<j$ we have that $tr.deg{_\mathcal{C}}\mathcal{C}(x_i,y_i,x_j,y_j)=2$. We direct the reader to the papers \cite[Section 5]{CasaleFreitagNagloo} or \cite[Section 2]{FreitagJaouiMarkerNagloo} where this is discussed in detail.
\end{rem}
It can be quite difficult to prove that a planar system (or any differential equation) is strongly minimal. The following is the complete list of equations of order $>1$ and defined over $\mathbb{C}$ that have been shown to be strongly minimal (and hence are geometrically trivial)
\begin{enumerate}
    \item Poizat showed that the equation $\frac{X''}{X'}=\frac{1}{X}$ is strongly minimal (cf. \cite{Marker} for an explanation). This was generalized first by Brestovski \cite{Brestovski} and later by Freitag, Jaoui, Marker and Nagloo \cite{FreitagJaouiMarkerNagloo}, where in both cases the rational function $\frac{1}{X}$ is replaced by rational functions of certain specific forms.
    \item Freitag and Scanlon \cite{FreitagScanlon}, have showed that the third order algebraic differential equation satisfied by the modular $j$-function is strongly minimal. The work of Blázquez-Sanz, Casale, Freitag and Nagloo (\cite{BlazquezCasaleFreitagNagloo2},\cite{BlazquezCasaleFreitagNagloo}, \cite{CasaleFreitagNagloo}) generalizes their result to the case of the uniformizers of Fuchsian groups of the first kind and to other generic Schwarzian triangle equations.
    \item A spectacular result of Jaoui (\cite{Jaoui}) proves that all generic planar systems of degree $\geq 2$, that is those where {\em all} the complex coefficients are algebraically independent over $\mathbb{Q}$ and so that the defining polynomials are of degree $\geq 2$, are strongly minimal.
\end{enumerate}

One of the main difficulty in proving that a given planar system $S$ is strongly minimal, is the need to check that there are no $D_S$-invariant  curve over {\em all} differential fields extension of $\mathcal{C}$. Moreover, a recent result of Freitag, Jaoui and Moosa \cite{FreitagJaouiMoosa} simplifies this task quite drastically. We state their result in a form that is suitable for our work and direct the reader to their paper for the more general statement.
\begin{fact}\label{FreitagJaouiMoosa}
Suppose that a planar system $S$ over $\mathcal{C}$ is not strongly minimal. Then there exists a $D_S$-invariant  curve over $\mathcal{C}(x,y)$, where $(x,y)$ is $\mathcal{C}$-generic solution of $S$.    
\end{fact}
This result will play a crucial role in our proof of strong minimality of the Lotka-Volterra system. Finally, let us point out that, except for the third case of the generic planar systems, in all the above work proving strong minimality of specific differential equations, one can also find a classification of the algebraic relations that exist among the solutions of the relevant equations. Indeed, in this sense establishing strong minimality (and geometric triviality) can be seen as a strategy for obtaining functional transcendence results. The following result of Brestovski \cite[Theorem 2]{Brestovski} which we will use in the next section, gives an old evidence of this idea for differential equations of a very specific form.
\begin{fact}\label{Brestovski}
Let $F\in \mathcal{C}\{X\}\setminus\mathcal{C}$ be a differential polynomial, $G_1,\ldots,G_n\in \mathcal{C}\gen{X}\setminus\mathcal{C}$ be differential rational functions and $a_1,\ldots,a_n\in\mathcal{C}$ be linearly independent over $\mathbb{Q}$. Assume that the set $\mathcal{X}$ defined by the equation
$$F'=\sum^n_{i=1}a_i\frac{G'_i}{G_i}$$
has $\text{ord}(\mathcal{X})=2$ and is strongly minimal. Then if $x_1,x_2\in\mathcal{X}$ are such that $x_2\in\mathcal{C}\gen{x_1}^{alg}$, then there is $m\in\mathbb{N}\setminus\{0\}$ such that $F(x_1)^m=F(x_2)^m$.
\end{fact}
\begin{rem}
 Notice that an equation of Brestovski form automatically has an elementary first integral, namely the function $H:=F-\sum^n_{i=1}a_i \text{log}\;G$ since $DH(X,X')=0$ for all solutions. Using the work of Prelle and Singer \cite{PrelleSinger}, it can be shown that a rational autonomous planar system (that is one where $f,g\in\mathbb{C}(X,Y)$) has an elementary first integral if and only if it is in definable bijection with an equation of Brestovski form with $F,G_1,\ldots,G_n$ {\em algebraic} over $\mathbb{C}(X,X')$.
\end{rem}
Before we move on to the proofs of our results let us mention the $\omega$-categoricity conjecture for the $\mathbb{C}$-definable strongly minimal sets of second order (i.e., including all strongly minimal autonomous planar systems).
\begin{defn}
Let $\mathcal{X}$ be a strongly minimal set and assume that $\mathcal{X}$ is definable over $\mathcal{C}$. If $\text{ord}(\mathcal{X})=1$, we further assume that $\mathcal{X}$ is geometrically trivial. By Fact \ref{SMimpliesGT}, this automatically holds if $\text{ord}(\mathcal{X})>1$. We say that $\mathcal{X}$ is {\em $\omega$-categorical} if for any non-constant solution\footnote{We only need one solution here since $\mathcal{X}$ is geometrically trivial.} $x\in \mathcal{X}$, only finitely many other solutions of $\mathcal{X}$ are in $\mathcal{C}\gen{x}^{alg}$, i.e., $\mathcal{C}\gen{x}^{alg}\cap \mathcal{X}$ is a finite set.
\end{defn}

\begin{conj}\label{conj}
Let $\mathcal{X}$ be a strongly minimal set with $\text{ord}(\mathcal{X})=2$ and assume that $\mathcal{X}$ is definable over $\mathcal{C}$. Then the following equivalent conditions hold:
\begin{enumerate} 
\item $\mathcal{X}$ is $\omega$-categorical.
\item There is $k\in\mathbb{N}\setminus\{0\}$ such that if $x_1,\ldots,x_n$ are distinct non-constant solutions of $\mathcal{X}$ with $$tr.deg_{\mathcal{C}}\mathcal{C}\gen{x_1\ldots,x_n}=2n,$$
then for any other non-constant solution $x$ of $\mathcal{X}$, except for $kn$ many,
$$tr.deg_{\mathcal{C}}\mathcal{C}\gen{x_1,\ldots,x_n,x}=2n+2.$$
\end{enumerate}
\end{conj}
It was originally conjectured/believed that such a property would hold for geometrically trivial strongly minimal sets of {\em any order} in differentially closed fields. Hrushovski \cite{Hrushovski} showed that this holds for order $1$ geometrically trivial definable sets. The work of Freitag and Scanlon \cite{FreitagScanlon} gave a counterexample to the general conjecture using the third order differential equation satisfied by the $j$-function. The work of Casale, Freitag and Nagloo \cite{CasaleFreitagNagloo} seems to suggest that the property only fails for equations related to arithmetic Fuchsian groups. Moreover, all known counterexamples are of order $3$, which motivates the above conjecture.

Let us give a proof that the conditions in Conjecture \ref{conj} are equivalent.
\begin{proof}[Proof that $(1)\iff (2)$ in Conjecture \ref{conj}]
Clearly if $(2)$ holds, then $\mathcal{X}$ is $\omega$-categorical. So assume $\mathcal{X}$ is $\omega$-categorical. By definition, for any $x\in\mathcal{X}$ non-constant we have that $\mathcal{C}\gen{x}^{alg}\cap \mathcal{X}$ is a finite set, say $|\mathcal{C}\gen{x}^{alg}\cap \mathcal{X}|=k$ for some $k\in\mathbb{N}\setminus\{0\}$. It is not hard to see, using strong minimality \footnote{It is a general fact that $\omega$-categorical strongly minimal sets are unimodular (cf. \cite[Remark 2.4.2]{PillayGST})}, that the same $k$ works for any non-constant solution in $\mathcal{X}$, i.e., if $y\in\mathcal{X}$ is non-constant, then $|\mathcal{C}\gen{y}^{alg}\cap \mathcal{X}|=k$. 

Let $x_1,\ldots,x_n$ be distinct non-constant solutions of $\mathcal{X}$ and assume that $tr.deg_{\mathcal{C}}\mathcal{C}\gen{x_1\ldots,x_n}=2n$.  Let $z\in\mathcal{X}$ be non-constant such that $tr.deg_{\mathcal{C}}\mathcal{C}\gen{x_1,\ldots,x_n,z}\neq2n+2.$ Using strong minimality of $\mathcal{X}$, we get that $z\in \mathcal{C}\gen{x_1,\ldots,x_n}^{alg}$ since $tr.deg_{\mathcal{C}}\mathcal{C}\gen{x_1,\ldots,x_n,z}=2n$. Hence all we need to show is that $\mathcal{C}\gen{x_1,\ldots,x_n}^{alg}$ contains exactly $kn$ many distinct non-constant solutions from $\mathcal{X}$. 

The above observation gives that $|\mathcal{C}\gen{x_i}^{alg}\cap \mathcal{X}|=k$ for each  $1\leq i\leq n$. Furthermore,  if $z\in \mathcal{C}\gen{x_1,\ldots,x_n}^{alg}$, then geometrically trivial (using Fact \ref{SMimpliesGT}) gives that there is $x_i$, $1\leq i\leq n$, such that $z\in \mathcal{C}\gen{x_i}^{alg}$. So $\mathcal{C}\gen{x_1,\ldots,x_n}^{alg}$ has at most $kn$ many distinct non-constant solutions from $\mathcal{X}$. On the other hand, the assumption $tr.deg_{\mathcal{C}}\mathcal{C}\gen{x_1\ldots,x_n}=2n$ guarantees that if $z\in \mathcal{C}\gen{x_i}^{alg}\cap \mathcal{X}$ for some $1\leq i\leq n$, then $z\not\in \mathcal{C}\gen{x_j}^{alg}$ for any $i\neq j$. From all this, we have that $(2)$ follows.
\end{proof}
While we are still quite far from proving the above conjecture, we do have the following:
\begin{thm}
Conjecture \ref{conj} holds of all order 2 strongly minimal set defined by an equation of Brestovski form.
\end{thm}
\begin{proof}
Let $\mathcal{X}$ be such a strongly minimal set and let $x\in \mathcal{X}$ be nonconstant. Since $\mathcal{X}$ has order $2$, we have that $F(x)'\in\mathcal{C}\gen{x}=\mathcal{C}(x,x',x'')$ and so there are differential polynomials $R,S\in\mathcal{C}\{X\}$ of order at most $2$, such that $$F(x)'=\frac{R(x)}{S(x)}.$$ This is of course just a rewriting of the Brestovski equation but from this form we can conclude that $F$ is of order at most $1$. Let us write $P\in\mathcal{C}\{X\}$ for second order differential polynomial $$P(X):=F(X)'S(X)-R(X).$$ Notice that $\mathcal{X}$ is defined by $(P(X)=0)\wedge (S(X)\neq0)$. 

Let $u:=F(x)\in\mathcal{C}\gen{x}$ and let $Q\in\mathcal{C}\gen{x}\{X\}$ be the differential polynomial $$Q(X):=u'S(X)F(X)-uR(X).$$ Let $\mathcal{Y}$ be the set defined by $(P(X)=0)\wedge (Q(X)=0)\wedge (S(X)\neq0)$.
\vspace{.1in}

\noindent {\bf Claim:} Let $y\in\mathcal{X}\setminus\mathcal{C}$. Then $y\in\mathcal{Y}$ if and only if $y\in \mathcal{C}\gen{x}^{alg}$
\begin{proof}[Proof of Claim]
Let $y\in\mathcal{X}\setminus\mathcal{C}$. First note that $u=c F(y)$ for some $c\in\mathcal{C}$ if and only if $(\frac{u}{F(y)})'=0$ if and only if $u'F(y)-F(y)'u=0$ if and only if $y\in\mathcal{Y}$. 

So using Fact \ref{Brestovski}, we have that if $y\in \mathcal{C}\gen{x}^{alg}$, then $u=\epsilon F(y)$ for $\epsilon\in\mathcal{C}$ some $m^{th}$-root of unity, where $m\in\mathbb{N}\setminus\{0\}$. In particular $y\in\mathcal{Y}$. On the other hand, if $y\in\mathcal{Y}$, then there is a $c\in\mathcal{C}$ such that $u=cF(y)$. It then follows, since $F$ is of order $<2$, that $tr.deg_{\mathcal{C}\gen{x}}\mathcal{C}\gen{x,y}<2$ which implies by strong minimality that $y\in \mathcal{C}\gen{x}^{alg}$ (and so $c$ is a root of unity).
\end{proof}


By construction $\mathcal{Y}$ is a $\mathcal{C}\gen{x}$-definable subset of $\mathcal{X}$. We claim that $\mathcal{X}\setminus\mathcal{Y}$ is nonempty and  infinite. Indeed, from the above claim, we have that any $z\in\mathcal{X}$ with $z\not\in \mathcal{C}\gen{x}^{alg}$ is such that $z\not\in\mathcal{Y}$; so $\mathcal{X}\setminus\mathcal{Y}$ is nonempty. On the other hand, if $\mathcal{X}\setminus\mathcal{Y}$ were finite and $y\in\mathcal{X}\setminus\mathcal{Y}$, then finiteness would force $y\in \mathcal{C}\gen{x}^{alg}$, so that by the claim $y\in \mathcal{Y}$, a contradiction. 

Consequently, $\mathcal{X}$ must be $\omega$-categorical; otherwise, the set $\mathcal{C}\gen{x}^{alg}\cap \mathcal{X}$ would be infinite, which would imply that $\mathcal{Y}$ is an infinite co-infinite definable subset of $\mathcal{X}$, contradicting strong minimality.
\end{proof}

\section{Proofs of the theorems}\label{Section3}

Throughout we assume that $(\mathcal{U},D)$ is a universal differentially closed field with $\mathbb{C}$ as its fields of constants. We will also assume that $a,b,c,d\in\mathbb{C}^*$. Observe that the only constant solution of the classical Lotka-Volterra system
   \begin{equation}\notag
LV_{a,b,c,d}:=\begin{dcases}
    X'= X(aY + b) \\
    Y'= Y(cX + d)\\
     X\neq0 \\
    Y\neq0.\\
\end{dcases}
\end{equation} is $(x,y)=(-\frac{d}{c},-\frac{b}{a})$. We start by giving the proof of Theorem \ref{Theorem3} assuming Theorem \ref{Theorem1}, that is assuming that $LV_{a,b,c,d}$ is strongly minimal.

\begin{thm}[\bf Theorem \ref{Theorem3}]
 \sloppy Assume that $d$ and $b$ are linearly independent over $\mathbb{Q}$ and that $(x_1,y_1),\ldots,(x_n,y_n)$ are distinct non-constant solutions of the classical Lotka-Volterra system $LV_{a,b,c,d}$. Then $$tr.deg_{\mathbb{C}}\mathbb{C}(x_1,y_1,\ldots,x_n,y_n)=2n.$$
\end{thm}

\begin{proof}
Note first that using the transformation $x\mapsto cx,y\mapsto ay$, we have that the sets defined by $LV_{a,b,c,d}$ and $LV_{1,b,1,d}$ are in bijection. So, without loss of generality, we can assume that $a=c=1$.  Furthermore, using Theorem \ref{Theorem1}, more precisely using geometric triviality, we only need to prove the theorem for the case $n=2$. 

Arguing by contradiction, assume that $(x,y)$ and $(u,v)$ are two solutions of $LV_{1,b,1,d}$ such that $u,v\in\mathbb{C}(x,y)^{alg}$. Using geometric triviality once more  (see Remark \ref{WeaklyOrthogonal}), we get that $u,v\in\mathbb{Q}(b,d)(x,y)^{alg}$. Let us write $\mathcal{C}$ for the field $\mathbb{Q}(b,d)^{alg}$ and also let $z:=x-y$. We have that $\mathcal{C}\gen{z}=\mathcal{C}(x,y)$. Indeed it is not hard to see that 
\begin{eqnarray*}
&(b-d)x=z'-dz\; \\
&(b-d)y=z'-bz.
\end{eqnarray*}
Furthermore, a simple computation shows that 
$$z'=x'-y'=b\frac{y'}{y}-d\frac{x'}{x}.$$
The same expressions hold for $w:=u-v$ (i.e., also replacing $(x,y)$ with $(u,v)$ in the above expressions). So, we have that the second order algebraic differential equation satisfied by $z$ and $u$ can be written as
$$X'=b\frac{(X'-bX)'}{X'-bX}-d\frac{(X'-dX)'}{X'-dX}.$$
Since $\mathcal{C}\gen{z}=\mathcal{C}(x,y)$ and by Corollary \ref{CorollaryLV}, all solutions of $LV_{1,b,1,d}$ are $\mathcal{C}$-generic, we have that the set defined by this equation is also strongly minimal. Setting $F:=X$, $G_1:=X'-bX$ and $G_2:=X'-dX$ (so $F\in \mathcal{C}\{X\}\setminus \mathcal{C}$ and $G_1,G_2\in \mathcal{C}\gen{X}\setminus\mathcal{C}$), we have that the equation satisfied by $z$ and $w$ is of Brestovski form
$$F'=b\frac{G'_1}{G_1}-d\frac{G'_2}{G_2}.$$
Since $w\in\mathcal{C}\gen{z}^{alg}$ (recall $u,v\in\mathcal{C}(x,y)^{alg}$ and $\mathcal{C}\gen{z}=\mathcal{C}(x,y)$) and $\frac{d}{b}\not\in\mathbb{Q}$, using Fact \ref{Brestovski}, we have that $F(z)^m=F(w)^m$ for some $m\in\mathbb{N}\setminus\{0\}$. In other words, for some $m\in\mathbb{N}\setminus\{0\}$,
$$(x-y)^m=(u-v)^m.$$
Hence, it follows that for some $m^{th}$-root of unity $\epsilon\in\mathbb{C}$ \[x-y=\epsilon(u-v).\] Differentiating and using the $LV_{1,b,1,d}$ equation, we get that \[bx-dy=\epsilon bu-\epsilon dv.\] Using the last two equations, it can be easily computed that $x=\epsilon u$ and $y=\epsilon v$. Since both $(x,y)$ and $(u,v)$ are solutions of $LV_{1,b,1,d}$, this forces $\epsilon=1$, contradicting that $(x,y)$ and $(u,v)$ are distinct.
\end{proof}

It hence remains to establish Theorem \ref{Theorem1} and Theorem \ref{Theorem2}. 

\subsection{The classical Lotka-Volterra system} 
We assume from now until the end of the paper that $t\in\mathcal{U}$ is a fixed element such that $Dt=1$. For our proof, the following (folklore) result will be repeatedly used. A proof can be found in \cite[Page 42]{Matsuda}, but we give it here for completeness.
\begin{fact} \label{AlgebraicSolution1} Let $\mathcal{C}$ be an algebraically closed subfield of $\mathbb{C}$ and let $F\in \mathcal{C}(t)$. Then the equation $y'=Fy$ has a non-zero solution in $\mathcal{C}(t)^{alg}$ if and only if the partial fraction decomposition of $F$ has the form $F(t) = \sum_{i}\frac{e_i}{t-a_i}$, where $a_i\in\mathcal{C}$, and such that $e_i\in\mathbb{Q}$.
\end{fact}
\begin{proof}
If $F(t) = \sum_{i}\frac{e_i}{t-a_i}$, where $a_i\in\mathcal{C}$ and $e_i\in\mathbb{Q}$, then it is not hard to see that $f(t)=\prod_i(t-a_i)^{e_1}$ is an algebraic solution of $y'=Fy$. 

Let $F\in \mathcal{C}(t)$ and assume that $f\in\mathcal{C}(t)^{alg}$ is a solution of $y'=Fy$. Let $G\in\mathcal{C}[X,Y]$ be an irreducible polynomial such that $G(t,f)=0$. We may assume that $G=\sum^{s}_{i=0}A_iY^i$ with $A_i\in\mathcal{C}[X]$ such that $A_0,A_s\neq0$.  Differentiating $G(t,f)=0$ gives
\[\sum^{s}_{i=0}A'_if^i+\sum^{s}_{i=1}iA_if^{i-1}f'=0.\]
Using $f'=Ff$ and simplifying we get 
$$0=\sum^{s}_{i=0}A'_if^i+\sum^{s}_{i=1}iFA_if^{i}=\sum^{s}_{i=0}(A'_i+iFA_i)f^{i}$$
Let $P\in\mathcal{C}[X,Y]$ be the polynomial such that 
\[P(t,f)=\sum^{s}_{i=0}(A'_i+iFA_i)f^{i}=0.\]

Since $G$ is irreducible, $G$ and $P$ have the same degree in $Y$ and every roots of $F(t,Y)$ is one of $P(t,Y)$, there is a polynomial $H\in\mathcal{C}[X]$ such that $P=HG$. Comparing coefficients of $f^i$ in $P(t,f)=H(t)G(t,f)$ gives
$\frac{A'_i+iFA_i}{A_i}=\frac{A'_0}{A_0}=H(t)$ for all $0\leq i\leq s$ for which $A_i\neq0$. In particular, $\frac{A'_s+sFA_s}{A_s}=\frac{A'_0}{A_0}$ giving $$F=\frac{1}{m}\left(\frac{A'_0}{A_0}-\frac{A'_s}{A_s}\right)$$ which is the desired form.
\end{proof}

Let us begin with the proof of Theorem \ref{Theorem1}. Recall that for a field $K$,  $K\left(\left(X\right)\right)$ denotes the field of formal Laurent series in variable $X$, while $K\gen{\gen{X}}$ denotes the field of formal Puiseux series, i.e., the field $\bigcup_{d\in{\mathbb{N}}}K\left(\left(X^{1/d}\right)\right)$. We will also make use of the fact that if $K$ is an algebraically closed field of characteristic zero, the field $K\gen{\gen{X}}$ is also algebraically closed (cf. \cite[Corollary 13.15]{Eisenbud}).

As we saw in the proof of Theorem \ref{Theorem3}, all we have to do is prove the result for $LV_{1,b,1,d}$. Furthermore, if we let $\delta=\frac{1}{b}D$ and consider the transformation $x\mapsto \frac{1}{b}x,y\mapsto \frac{1}{b}y$,  we obtain the system
\begin{equation}\notag
\begin{dcases}
    \delta X= X(Y + 1) \\
    \delta Y= Y(X + d/b)\\
     X\neq0 \\
    Y\neq0.\\
\end{dcases}
\end{equation}
It hence suffices to show that the system $LV_{1,1,1,\frac{d}{b}}$ is strongly minimal. Throughout this subsection we write write $\mathcal{C}$ for the field $\mathbb{Q}(b,d)^{alg}$. We also let $\alpha=\frac{d}{b}$.
\begin{lem}\label{LemmaLV}
Assume $K$ is differential field extension of  $\mathcal{C}$ and $(x,y)$ a solution of $LV_{1,1,1,\alpha}$ such that
\[tr.deg_KK(x,y)=1.\]
Then there exists $a\in K^{alg}$ such that $a$ is a non-zero solution of the equation $$a'=ma,$$ where $m=1-r\alpha$ for some $r\in\mathbb{N}$.
 \end{lem}
 
 \begin{proof}
\sloppy     Assume $K$ and $(x,y)$ are as in the hypothesis of the lemma. Since $tr.deg_KK(x,y)=1$, we have that $x\in K (y)^{alg} $. So we can view  $x$ as an element of the field of Puiseux series $K^{alg} \gen{\gen{y}}, i.e., $ 
\[ x = \sum_{i=0}^{\infty} a_iy^{r+\frac{i}{e}},    a_i\in K^{alg}\]
with $a_0 \neq 0$.
Differentiating we get
\[ x'=\sum_{i=0}^{\infty} a'_iy^{r+\frac{i}{e}}+ \sum_{i=0}^{\infty} a_i({r+\frac{i}{e}})y^{r+\frac{i}{e}-1}y'.\]
Substituting $x'=xy+x$ and $y'=yx+\alpha y$ and simplifying gives
\[ \sum_{i=0}^{\infty} a_iy^{r+\frac{i}{e}+1} + \sum_{i=0}^{\infty} a_iy^{r+\frac{i}{e}}=\sum_{i=0}^{\infty} a_i'y^{r+\frac{i}{e}} + \sum_{i=0}^{\infty} (r+\frac{i}{e})a_ia_jy^{2r+\frac{i+j}{m}} + \sum_{i=0}^{\infty} (r+\frac{i}{e})\alpha a_iy^{r+\frac{i}{e}}\]

\noindent Noticing that $a_i'\in K$ for all $i$, we can compare the coefficients and  get the following three cases
\begin{itemize}
\item If $r<0$, by looking at the coefficients of $y^{2r}$ we get that $ra_0^2 = 0$. This contradicts the fact that $a_0 \neq 0$
\item If $r=0$, by looking at the coefficients of $y^0$ we get $a_0'=a_0$.
\item If $r>0$, by looking at the coefficients of $y^r$ we get $a_0'=(1-r\alpha)a_0$.
\end{itemize}
We are done. 
\end{proof} 

Using Fact \ref{AlgebraicSolution1} we get the following.

\begin{cor}\label{CorollaryLV}
The Classical Lotka-Volterra system $LV_{a,b,c,d}$ has no invariant curves over $\mathbb{C}$ or $\mathbb{C}(t)$. In particular, it has no solutions in $\mathbb{C}(t)^{alg}$.
\end{cor}
We are now ready to prove Theorem \ref{Theorem1}

\begin{thm}[\bf Theorem \ref{Theorem1}]\label{TheoremA}
Let $a,b,c,d\in\mathbb{C}^*$ and assume that $d$ and $b$ are linearly independent over $\mathbb{Q}$. Then the classical Lotka-Volterra system
\begin{equation}\notag
\begin{dcases}
    X'= aXY + bX \\
    Y'= cXY + dY\\
     X\neq0 \\
    Y\neq0\\
\end{dcases}
\end{equation}
is strongly minimal and hence geometrically trivial. 
\end{thm}

\begin{proof}
Arguing by contradiction, assume that $LV_{a,b,c,d}$ and hence $LV_{1,1,1,\alpha}$ (with $\alpha=\frac{d}{b}$) is not strongly minimal. So for some differential field $K$ extending $\mathcal{C}$ and some solution $(u,v)$ of $LV_{1,1,1,\alpha}$, we have that $tr.deg_KK(u,v)=1$. Using Fact \ref{FreitagJaouiMoosa}, we may assume that $K=\mathcal{C}(x,y)$ for some other solution $(x,y)$ of $LV_{1,1,1,\alpha}$. 

Using Lemma \ref{LemmaLV}, there is $a\in \mathcal{C}(x,y)^{alg}$ such that $a$ is a non-zero solution of the equation $a'=ma,$ where $m=1-r\alpha$ for some $r\in\mathbb{N}$. From Corollary \ref{CorollaryLV}, it follows that $a\not\in\mathcal{C}(x)^{alg}$. Indeed, since $\mathcal{C}(a)^{alg}$ is a differential field (i.e., closed under the derivation), the condition $a\in\mathcal{C}(x)^{alg}$ would imply $\mathcal{C}(a)^{alg}=\mathcal{C}(x)^{alg}$ and hence $y=\frac{x'}{x}-1\in\mathcal{C}(x)^{alg}$. Since $a\in \mathcal{C}(x,y)^{alg}$, we can view  $a$ as a non-constant element of the field of Puiseux series $\mathcal{C}(x)^{alg}\gen{\gen{y}}$, i.e., 
\[ a = \sum_{i=0}^{\infty} b_iy^{k+\frac{i}{e}},    b_i\in \mathcal{C}(x)^{alg}\]
with $b_0 \neq 0$. Differentiating and using $a'=ma$ we get
\[ ma=\sum_{i=0}^{\infty} b'_iy^{k+\frac{i}{e}}+ \sum_{i=0}^{\infty} a_i({k+\frac{i}{e}})y^{k+\frac{i}{e}-1}y'.\]
Using $y'=y(x+\alpha)$ and the fact that
\[b_i'=(xy+x)\frac{\partial b_i}{\partial x}=xy\frac{\partial b_i}{\partial x}+x\frac{\partial b_i}{\partial x}\]
we get
\[\sum_{i=0}^{\infty} mb_iy^{k+\frac{i}{e}}=\sum_{i=0}^{\infty} \left(x\frac{\partial b_i}{\partial x}\right)y^{k+\frac{i}{e}+1}+\sum_{i=0}^{\infty} \left(x\frac{\partial b_i}{\partial x}\right)y^{k+\frac{i}{e}}+ \sum_{i=0}^{\infty} b_i({k+\frac{i}{e}})(x+\alpha)y^{k+\frac{i}{e}}\]

Using Fact \ref{AlgebraicSolution1}, it follows that $k=0$. Indeed, if $k\neq0$, then we can look at the coefficients of $y^k$ and get that $mb_0=x\frac{\partial b_0}{\partial x}+b_0k(x+\alpha)$. In other words, $b_0\in\mathcal{C}(x)^{alg}$ satisfies
\[\frac{\partial b_0}{\partial x}=\left(\frac{m-r\alpha}{x}-k\right)b_0,\]
which is a contradiction. So $k=0$ and we can look at the coefficients of $y^0$ and get that $mb_0=x\frac{\partial b_0}{\partial x}$, that is $b_0\in\mathcal{C}(x)^{alg}$ satisfies
\[\frac{\partial b_0}{\partial x}=\frac{m}{x}b_0.\]
Recalling that $m=1-r\alpha$ for some $r\in\mathbb{N}$, since $\alpha=\frac{d}{b}\not\in\mathbb{Q}$, Fact \ref{AlgebraicSolution1} forces $r=0$ and hence $m=1$, i.e., $b_0$ is a solution of $\frac{\partial v}{\partial x}=\frac{1}{x}v$ in $\mathcal{C}(x)^{alg}$. 

Let $\beta\in\mathcal{C}$, $\beta\neq0$, be such that\footnote{All algebraic solutions have this form.} $b_0=\beta x$. From 
\[ a = \sum_{i=0}^{\infty} b_iy^{\frac{i}{e}}\]
we get that for some $\ell>0$
\[ a-\beta x= \sum_{i=\ell}^{\infty} b_iy^{\frac{i}{e}}.\]
Differentiating and using $a'=a$, and $x'=xy+x$ we get after some simplification that
\[(a-\beta x)-\beta xy=\sum_{i=\ell}^{\infty} b'_iy^{\frac{i}{e}}+\sum_{i=\ell}^{\infty} b_i\frac{i}{e}y^{\frac{i}{e}-1}y'.\]
One can use $a-\beta x= \sum_{i=\ell}^{\infty} b_iy^{\frac{i}{e}}$ and again apply $b_i'=xy\frac{\partial b_i}{\partial x}+x\frac{\partial b_i}{\partial x}$ and $y'=y(x+\alpha)$ to get
\[\sum_{i=\ell}^{\infty} b_iy^{\frac{i}{e}}-\beta xy=\sum_{i=\ell}^{\infty} x\frac{\partial b_i}{\partial x}y^{\frac{i}{e}+1}+\sum_{i=\ell}^{\infty} x\frac{\partial b_i}{\partial x}y^{\frac{i}{e}}+\sum_{i=\ell}^{\infty} b_i\frac{i}{e}(x+\alpha)y^{\frac{i}{e}}.\]
Comparing coefficients gives two distinct cases.\\

\noindent{\bf Case 1:} $\frac{\ell}{e}\neq 1$.\\

\noindent If $\frac{\ell}{e}> 1$ then comparing the coefficients of $y$ forces $\beta x=0$. But $\beta\neq 0$ and $x\neq 0$. Hence $\frac{\ell}{e}< 1$. However, in that case, by comparing the coefficients of $y^{\frac{\ell}{e}}$, we obtain that $b_{\ell}=x\frac{\partial b_{\ell}}{\partial x}+ b_{\ell}\frac{\ell}{e}(x+\alpha)$. In other words, $b_{\ell}\in\mathcal{C}(x)^{alg}$ is a solution of
$$\frac{\partial v}{\partial x}=\left(\frac{1-\frac{\alpha \ell}{e}}{x}-\frac{\ell}{e}\right)v.$$ Since $\ell>0$, this contradict Fact \ref{AlgebraicSolution1}.\\

\noindent{\bf Case 2:} $\frac{\ell}{e}= 1$.\\

\noindent In this case we look at the coefficients of $y$ and get that
\[b_{\ell}-\beta x=x\frac{\partial b_{\ell}}{\partial x}+b_{\ell}(x+\alpha).\]
So $b_{\ell}\in\mathcal{C}(x)^{alg}$ is a solution of
$$\frac{\partial v}{\partial x}=\left(\frac{1-\alpha}{x}-1\right)v-\beta.$$ The following lemma gives the desired result.
\end{proof}
\begin{lem}\label{LemmaLast1}
If $q\not\in\mathbb{Q}$ and $c_1,c_2\in\mathcal{C}$, then the inhomogeneous linear differential equation 
$$v'=\left(\frac{q}{t}+c_1\right)v+c_2$$
has no algebraic solutions in $\mathcal{C}(t)^{alg}$.
\end{lem}
\begin{proof}
The following proof is inspired by that of Fact \ref{AlgebraicSolution1}. For simplicity let us write $Q=\frac{q}{t}+c_1$. So $Q\in\mathcal{C}(t)$. Arguing by contradiction, assume that $f\in\mathcal{C}(t)^{alg}$ is a solution of $v'=Qv+c_2$ (recall $t\in \mathcal{U}$ is such that $t'=1$). Let $F\in\mathcal{C}[X,Y]$ be an irreducible polynomial such that $F(t,f)=0$. We may assume that $F=\sum^{s}_{i=0}A_iY^i$ with $A_i\in\mathcal{C}[X]$ such that $A_s\neq0$.  Differentiating $F(t,f)=0$ gives
\[\sum^{s}_{i=0}A'_if^i+\sum^{s}_{i=1}iA_if^{i-1}f'=0.\]
Using $f'=Qf+c_2$ and simplifying we get 
\[\sum^{s}_{i=0}A'_if^i+\sum^{s}_{i=1}iQA_if^{i}+\sum^{s}_{i=1}ic_2A_if^{i-1}=0.\]
Let $P\in\mathcal{C}[X,Y]$ be the polynomial such that 
\[P(t,f)=\sum^{s}_{i=0}A'_if^i+\sum^{s}_{i=1}iQA_if^{i}+\sum^{s}_{i=1}ic_2A_if^{i-1}=0\]
Note that if $P=\sum^{s}_{i=0}B_iY^i$ for $B_i\in\mathcal{C}[X]$, then $B_s=A'_s+sQA_s$. As in the proof of Fact \ref{AlgebraicSolution1} above, since $F$ is irreducible, $F$ and $P$ have the same degree in $Y$ and every roots of $F(t,Y)$ is one of $P(t,Y)$, there is a polynomial $G\in\mathcal{C}[X]$ such that $P=GF$. Comparing coefficients of $f^s$ in $P(t,f)=G(t)F(t,f)$ gives
$$B_s:=A'_s+sQA_s=GA_s.$$
In other words, using $Q=\frac{q}{z}+c_1$, we get $A_s$ is a nonzero solution in $C(t)^{alg}$ of the equation
$$v'=\left(G-sc_1-\frac{sq}{t}\right)v.$$
Regardless of whether $G(t)=sc_1$, since $q\not\in\mathbb{Q}$ this contradicts Fact \ref{AlgebraicSolution1}.
\end{proof}
\subsubsection{{The case $\frac{d}{b}\in\mathbb{Q}$}}\label{subsection}
Let us point out that if $b=d$, then the classical Lotka-Volterra system $LV_{1,b,1,b}$ (and hence $LV_{a,b,c,b}$)
\begin{equation}\notag
\begin{dcases}
    X'= XY + bX \\
    Y'= XY + bY\\
\end{dcases}
\end{equation}

\noindent is not strongly minimal. Indeed in this case, if $(x,y)$ is a solution of the system, letting $z:=x-y$ we see that $z'=bz$. Hence, the field $K=\mathbb{Q}(b)(z)$ is a differential field. For any other solution $(u,v)$ since $(u-w)'=b(u-w)$, it follows that for some constant $c\in\mathbb{Q}(b)$, we have that $u-v=cz$. In other words, for each $c\in\mathbb{Q}(b)$ the polynomial $P_c:=X-Y-cz\in K[X,Y]$ is an invariant polynomial over $K$:
$$P^D+X(Y+b)\frac{\partial P}{\partial X}+Y(X+b)\frac{\partial P}{\partial Y}=bP.$$ Our techniques does not allow us to prove that $LV_{a,b,c,d}$ is strongly minimal when $b\neq d$. As discussed in the Acknowledgment, Duan, Eagles and Jimenez \cite{DuanEaglesJimenez}, using an idea of Jaoui, have recently proved that the case $b=d$ is the only instance where strong minimality fails.

\subsection{The 2d-Lotka-Volterra system} 
Now we prove Theorem \ref{Theorem2}. Consider the 2d-Lotka-Volterra system
 \begin{equation}\notag
LV^{2d}_{a,b,c,d}:=\begin{dcases}
    X'= X(aY + b) \\
    Y'= Y(cX + dY)\\
     X\neq0 \\
    Y\neq0.\\
\end{dcases}
\end{equation} 
It has one constant solution, namely $(x,y)=(\frac{db}{ca},-\frac{b}{a})$. As in the previous section, using the transformation $x\mapsto cx,y\mapsto ay$, we have that the sets defined by $LV^{2d}_{a,b,c,d}$ and $LV^{2d}_{1,b,1,\frac{d}{a}}$ are in bijection. Furthermore, if we let $\delta = \frac{1}{b}D$ and $x\mapsto \frac{1}{b}x,y\mapsto \frac{1}{b}y$, we get the system
\begin{equation}\notag
\begin{dcases}
    \delta X= X(Y + 1) \\
    \delta Y= Y(X + \frac{d}{a}Y)\\
     X\neq0 \\
    Y\neq0.\\
\end{dcases}
\end{equation}
It hence suffices to show that the system $LV^{2d}_{1,1,1,\frac{d}{a}}$ is strongly minimal. In this subsection we write $\gamma = \frac{d}{a}$, and let $\mathcal{C}$ be the field $\mathbb{Q}(a,d)^{alg}$.

\begin{lem}\label{LemmaLV6}
     Assume $K$ is a differential field extension of $\mathcal{C}$ and $(x,y)$ a solution of $LV^{2d}_{1,1,1,\gamma}$ such that
\[tr.deg_KK(x,y)=1.\]
Then there exists an element $a\in K^{alg}$ such that $a$ is a non-zero solution of the equation $$a'=a$$ 
 \end{lem}
    
\begin{proof}
\sloppy    Since $tr.deg_KK(x,y)=1$, we have that $x\in K (y)^{alg} $ and hence $x \in K^{alg} \gen{\gen{y}}, i.e., $ 
\[ x = \sum_{i=0}^{\infty} a_iy^{r+\frac{i}{e}},    a_i\in K^{alg}\]
with $a_0 \neq 0$.
Differentiating we get
\[ x'=\sum_{i=0}^{\infty} a'_iy^{r+\frac{i}{e}}+ \sum_{i=0}^{\infty} a_i({r+\frac{i}{e}})y^{r+\frac{i}{e}-1}y'.\]
Substituting $x'=xy+x$ and $y'=yx+\gamma y^2$ and simplifying gives
\[ \sum_{i=0}^{\infty} a_iy^{r+\frac{i}{e}+1} + \sum_{i=0}^{\infty} a_iy^{r+\frac{i}{e}}=\sum_{i=0}^{\infty} a_i'y^{r+\frac{i}{e}} + \sum_{i=0}^{\infty} (r+\frac{i}{e})a_ia_jy^{2r+\frac{i+j}{e}} + \sum_{i=0}^{\infty} (r+\frac{i}{e})\gamma a_iy^{r+\frac{i}{e}+1}.\]

\noindent As in the previous proofs, we are done by comparing the coefficients:
\begin{itemize}
\item If $r<0$, by looking at the coefficients of $y^{2r}$ we get that $ra_0^2 = 0$. This contradicts the fact that $a_0 \neq 0$
\item If $r=0$, by looking at the coefficients of $y^0$ we get $a_0'=a_0$.
\item If $r>0$, by looking at the coefficients of $y^r$ we get $a_0'=a_0$.
\end{itemize} 
\end{proof} 
\begin{cor} \label{corollaryLV6}
 The 2d-Lotka-Volterra system $LV^{2d}_{a,b,c,d}$ has no invariant curves over $\mathbb{C}$ or $\mathbb{C}(t)^{alg}$.
\end{cor}

Now the proof of Theorem \ref{Theorem2} follows quite similarly as that of Theorem \ref{Theorem1}.

\begin{thmx}[\bf Theorem \ref{Theorem2}]\label{TheoremB}
\sloppy Let $a,b,c,d\in\mathbb{C}^*$ and assume that $\frac{d}{b}\not\in\mathbb{Q}$. Then the 2d-Lotka-Volterra system
\begin{equation}\notag
\begin{dcases}
    X'= X(aY + b) \\
    Y'= Y(cX + dY)\\
    X\neq0 \\
    Y\neq0\\
\end{dcases}
\end{equation}
is strongly minimal and hence geometrically trivial.
\end{thmx}

\begin{proof}
\sloppy Arguing by contradiction, assume that $LV^{2d}_{a,b,c,d}$, and hence $LV^{2d}_{1,1,1,\beta}$, is not strongly minimal. So, for some differential field $K$, extending $\mathcal{C}$, and some solution $(u,v)$, we have that $tr.deg_KK(u,v) = 1$. Using Fact \ref{FreitagJaouiMoosa}, we may assume $K = \mathcal{C}(x,y)$, for some other distinct solution $(x,y)$ of $LV^{2d}_{1,1,1,\beta}$. As before, Lemma \ref{LemmaLV6} gives that there is a non-zero solution of the equation $a'=a$ in $\mathcal{C}(x,y)^{alg}$. 

Since $a\not\in\mathcal{C}(x)^{alg}$ and $a\in \mathcal{C}(x,y)^{alg}$, we can view  $a$ as a non-constant element of the field of Puiseux series $\mathcal{C}(x)^{alg}\gen{\gen{y}}$, i.e., 
\[ a = \sum_{i=0}^{\infty} b_iy^{k+\frac{i}{e}},    b_i\in \mathcal{C}(x)^{alg}\]
with $b_0 \neq 0$. Differentiating and using $a'=a$ we get
\[ a=\sum_{i=0}^{\infty} b'_iy^{k+\frac{i}{e}}+ \sum_{i=0}^{\infty} a_i({k+\frac{i}{e}})y^{k+\frac{i}{e}-1}y'.\]
Substituting $y'=y(x+\gamma y)$ and also using the fact that
\[b_i'=(xy+x)\frac{\partial b_i}{\partial x}=xy\frac{\partial b_i}{\partial x}+x\frac{\partial b_i}{\partial x}\]
we get
\[ \sum_{i=0}^{\infty}b_iy^{r+\frac{i}{e}} = \sum_{i=0}^{\infty} (x\frac{\partial b_i}{\partial x})y^{r+\frac{i}{e}+1} + \sum_{i=0}^{\infty} (x\frac{\partial b_i}{\partial x})y^{r+\frac{i}{e}}+ \sum_{i=0}^{\infty} b_ix(r+\frac{i}{e})y^{r+\frac{i}{e}}+\sum_{i=0}^{\infty} b_i \gamma (r+\frac{i}{e})y^{r+\frac{i}{e}+1} \]

Looking at the coefficients of $y^r$ and by using Fact \ref{AlgebraicSolution1}, we obtain that $r=0$. Indeed, if $r\neq0$, then it is not hard to see, by comparing the coefficients of $y^r$, that $b_0=x\frac{\partial b_0}{\partial x}+b_0rx$, which is a contradiction since $b_0 \in \mathcal{C}(x)^{alg}$. So $r=0$ and we can look at the coefficients of $y^0$ and get that $b_0=x\frac{\partial b_0}{\partial x}$, that is $b_0\in\mathcal{C}(x)^{alg}$ satisfies
\[\frac{\partial b_0}{\partial x}=\frac{1}{x}b_0.\]
In other words, we have $b_0 = \beta x$ for some $\beta \in \mathcal{C}$ with $\beta\neq 0$ and from the Puiseux series for $a$ we get that
\[ a-\beta x= \sum_{i=\ell}^{\infty} b_iy^{\frac{i}{e}}.\]
for some $l > 0$.
Differentiating and using $a'=a$, $x'=xy+x$, and $y'=y(x+\gamma y)$, we get (as in the previous proof)
\[\sum_{i=\ell}^{\infty} b_iy^{\frac{i}{e}}-\beta xy=\sum_{i=\ell}^{\infty} x\frac{\partial b_i}{\partial x}y^{\frac{i}{e}+1}+\sum_{i=\ell}^{\infty} x\frac{\partial b_i}{\partial x}y^{\frac{i}{e}}+\sum_{i=\ell}^{\infty} b_i\frac{i}{e}xy^{\frac{i}{e}}+\sum_{i=\ell}^{\infty} b_i\frac{i}{e}\gamma y^{\frac{i}{e}+1} .\]
By separating the case $\frac{\ell}{e}\neq 1$ versus $\frac{\ell}{e}=1$, once can eliminate the case $\frac{\ell}{e}\neq 1$ exactly as was argued for Theorem \ref{Theorem1} (with $\alpha=0$). So we are left with $\frac{\ell}{e}=1$.

\noindent In that case we look at the coefficients of $y$ and get that
\[b_{\ell}-\beta x=x\frac{\partial b_{\ell}}{\partial x}+b_{\ell}x.\]
So $b_{\ell}\in\mathcal{C}(x)^{alg}$ is a solution of
$$\frac{\partial v}{\partial x}=\left(\frac{1}{x}-1\right)v-\beta.$$ This time, we are done using the following lemma.
\end{proof}
\begin{lem}
Let $c\in\mathbb{C}^*$. The inhomogeneous linear differential equation 
$$v'=\left(\frac{1}{t}-1\right)v+c$$
has no algebraic solutions in $\mathcal{C}(t)^{alg}$.
\end{lem}
\begin{proof}
The proof is very similar to that of Lemma \ref{LemmaLast1}.  We write $Q=\frac{1}{t}-1$ and argue by contradiction that $f\in\mathcal{C}(t)^{alg}$ is a solution of $v'=Qv+c_2$. Using the irreducible polynomial $F=\sum^{s}_{i=0}A_iY^i\in\mathcal{C}[X,Y]$ such that $F(t,f)=0$, $A_i\in\mathcal{C}[X]$ and $A_s\neq0$, we obtain (after differentiation) a polynomial $P=\sum^{s}_{i=0}B_iY^i\in\mathcal{C}[X,Y]$ such that 
\[P(t,f)=\sum^{s}_{i=0}A'_if^i+\sum^{s}_{i=1}iQA_if^{i}+\sum^{s}_{i=1}icA_if^{i-1}=0\]
As before, using irreducibility of $F$ and the fact that $P$ have the same degree in $Y$, we get a polynomial $G\in\mathcal{C}[X]$ such that $P=GF$. After comparing coefficients of $f^s$ in $P(t,f)=G(t)F(t,f)$ we get
$$B_s:=A'_s+sQA_s=GA_s.$$
In other words, using $Q=\frac{1}{t}-1$, we get $A_s$ is a nonzero solution in $C(t)^{alg}$ of the equation
$$v'=\left(G+s-\frac{s}{t}\right)v.$$
If $G(t)+s\neq0$, then Fact \ref{AlgebraicSolution1} gives a contradiction. On the other hand when $G(t)+s=0$, we have that $A_s$ is a non-zero {\em polynomial} solution in $C(t)^{alg}$ of  $v'=\left(\frac{-s}{t}\right)v.$ Since $s\in\mathbb{N}\setminus\{0\}$ we have that all non-zero solutions in $C(t)^{alg}$ of this equation are of the form $v=\frac{d}{t^s}$ for some constant $d\in\mathbb{C}^*$. This is again a contradiction.
\end{proof}

\end{document}